\documentclass[reqno,twoside,12pt]{amsart}

\hoffset - 1.5cm

\usepackage{amsmath}
\usepackage{amssymb}
\usepackage{graphicx}
\usepackage{amstext}
\usepackage{amsopn}
\usepackage{amsfonts}
\usepackage{xcolor}
\usepackage{setspace}
\usepackage{enumerate}
\usepackage[polish,english]{babel}
\usepackage[utf8]{inputenc}
\usepackage{polski}
\usepackage{amsaddr}

\usepackage{lmodern}

\newtheorem{Remark}{Remark}[section]
\newtheorem{Corollary}[Remark]{Corollary}
\newtheorem{Definition}[Remark]{Definition}

\newtheorem{Lemma}[Remark]{Lemma}

\newtheorem{Theorem}[Remark]{Theorem}

\newcommand{\bH}{\mathbb{H}}
\newcommand{\bI}{\mathbb{I}}

\newcommand{\bN}{\mathbb{N}}

\newcommand{\bR}{\mathbb{R}}

\newcommand{\bU}{\mathbb{U}}
\newcommand{\bW}{\mathbb{W}}
\newcommand{\bV}{\mathbb{V}}

\newcommand{\bZ}{\mathbb{Z}}

\newcommand{\cC}{\mathcal{C}}

\newcommand{\cH}{\mathcal{H}}

\newcommand{\cL}{\mathcal{L}}
\newcommand{\cM}{\mathcal{M}}
\newcommand{\cN}{\mathcal{N}}

\newcommand{\cR}{\mathcal{R}}

\newcommand{\cT}{\mathcal{T}}

\newcommand{\cV}{\mathcal{V}}
\newcommand{\cW}{\mathcal{W}}

\numberwithin{equation}{section} \errorcontextlines=0

\newcommand{\degg}{\nabla_{G}\textrm{-}\mathrm{deg}}
\newcommand{\degh}{\nabla_{H}\textrm{-}\mathrm{deg}}
\newcommand{\im}{\mathrm{ im \;}}

\newcommand{\ds}{\displaystyle}

\newcommand{\h}{\mathbb{H}}

\begin{document}

\title[Bifurcations from degenerate orbits]{Bifurcations from degenerate orbits of solutions of nonlinear elliptic systems}

\author{Anna Go\l\c{e}biewska}
\address{Faculty of Mathematics and Computer Science
Nicolaus Copernicus University \\
ul. Chopina $12 \slash 18$, PL-87-100 Toru\'{n},
Poland, ORCID 0000--0002--2417--9960\\
Anna.Golebiewska@mat.umk.pl}


\author{Joanna Kluczenko}
\address{Faculty of Mathematics and Computer Science, University of Warmia and Mazury \\ul.~Sloneczna 54, PL-10-710 Olsztyn, Poland,
ORCID 0000--0003--0121--1136\\
jgawrycka@matman.uwm.edu.pl}

\author{Piotr Stefaniak}
\address{Faculty of Mathematics and Computer Science
Nicolaus Copernicus University \\
ul. Chopina $12 \slash 18$, PL-87-100 Toru\'{n},
Poland, ORCID 0000--0002--6117--2573\\
cstefan@mat.umk.pl}

\numberwithin{equation}{section}
\allowdisplaybreaks

\maketitle

\begin{abstract}
The aim of this paper is to study global bifurcations of non-constant solutions of some nonlinear elliptic systems, namely the system on a sphere and the Neumann problem on a ball. We study the bifurcation phenomenon from families of constant solutions given by critical points of the potentials. Considering this problem in the presence of additional symmetries of a compact Lie group, we study orbits of solutions and, in particular, we do not require  the critical points to be isolated. Moreover, we allow the considered orbits of critical points to be degenerate. To prove the bifurcation we compute the index of an isolated degenerate critical orbit in an abstract situation. This index is given in terms of the degree for equivariant gradient maps.
\end{abstract}

\section{Introduction}

The aim of this article is to study global bifurcations of solutions of two kinds of nonlinear systems of elliptic equations: the system
on a sphere
\begin{equation}\label{eq:sphere11}
- \Delta u =  \nabla_u F(u, \lambda) \quad \text{ on } \quad S^{N-1}
\end{equation}
and the Neumann problem on a ball
\begin{equation}\label{eq:ball11}
\left\{
 \begin{array}{rclcl}   -\triangle u & =& \nabla_u F(u,\lambda )   & \text{ in   } & B^N \\
                   \frac{\partial u}{\partial \nu} & =  & 0 & \text{ on } & S^{N-1}
\end{array}\right.
\end{equation}
with a potential $F$ satisfying some additional assumptions given in Section \ref{sec:elliptic} (the conditions (B1)-(B6)).

The global bifurcation problem concerns finding connected sets of solutions of some equations, emanating from a known family of so called trivial solutions. One of the most famous results on this topic is the Rabinowitz alternative, see \cite{Rabinowitz}, \cite{Rabinowitz1}. This result gives conditions to an occurrence of the phenomenon of global bifurcation and describes the basic structure of the emanating sets of solutions. It has been proved with the use of the Leray-Schauder degree. It is worth to point out that the proof relies on the properties of the degree, not its exact definition. Therefore, similar results can be obtained with the use of other degree theories satisfying analogous properties, in particular the property of generalised homotopy invariance (or otopy invariance), see for example \cite{BarKamNow}, \cite{GolRyb2011}, \cite{GuoLiu}, \cite{Mawhin}, \cite{Ryb2005milano}.

Rabinowitz's result and its later generalisations deal with an equation of the type $T(u,\lambda)=0$ having a family of trivial solutions of the form $\{u_0\}\times \bR$.
Such a situation appears naturally if the variational method of studying differential equations is applied,  i.e. solutions of the differential equation are associated with critical points of some functional, defined on a suitable Hilbert space.
Then the family of trivial solutions often consists of known constant functions and the bifurcation of non-constant solutions from this family is studied.

In such an approach it is required that there are $\widehat{\lambda}\in\bR$ such that the point $u_0$ is isolated in the set of solutions of the equation $T(u,\widehat{\lambda})=0$.
However, one can also consider a symmetric case, i.e. assume that $T$ is $G$-equivariant, where $G$ is a compact Lie group. Then if $u_0$ is a solution, all points from its orbit also solve the equation. Therefore, if the group has positive dimension, it is often impossible to find such $\widehat{\lambda}$.

Symmetries in differential problems come naturally from applications, in particular from mathematical physics. That is one of the reasons why  elliptic systems with different kinds of symmetries are presently studied by many mathematicians, see for example \cite{BalHooKraRa}, \cite{BonSeTi}, \cite{KaKo}, \cite{KaSo}, \cite{MaChenWa}, \cite{MaYan}. In these articles there are applied various methods: variational and non-variational.

A symmetric situation is also considered in this paper. Namely, the systems \eqref{eq:sphere11} and \eqref{eq:ball11} are defined on $SO(N)$-symmetric domains and we assume additionally that their potentials are $\Gamma$-invariant for $\Gamma$ being a compact Lie group. We study these problems using variational methods, investigating associated functionals. These functionals are $G=\Gamma\times SO(N)$-invariant, i.e. they inherit both kinds of symmetries of the problems. The family of trivial solutions is therefore of the form $G(u_0)\times\bR$ and we study global bifurcations from orbits of critical points (called critical orbits).

Methods of studying bifurcations from critical orbits have been recently developed. In \cite{PRS} there have been given formulae for an equivariant Conley index, allowing to study local bifurcation problems in the presence of orbits of solutions. These formulae, combined with the relation of the index with the equivariant degree theory, have been applied in \cite{GolKluSte} to obtain global bifurcations of solutions in Neumann problems on a ball.
In \cite{GolSte} we have proposed a different approach to study global bifurcations with a use of an equivariant degree at a neighbourhood of an orbit. As an application of these results we have studied elliptic systems on symmetric domains. In \cite{GolSte2} we have used these methods to study the existence of unbounded connected sets of solutions of systems on a sphere.
 In the papers mentioned above we have assumed that the isotropy groups of the critical points of the potentials are trivial. However, we have also considered the case without this assumption, see \cite{GolRybSte}.

As far as we know these are the only articles about elliptic systems where the isolation of the critical points of the potentials has not been required. In all of these papers we have assumed that the family of trivial solutions is given by critical orbits of the potentials of the systems, requiring additionally that these critical orbits are non-degenerate.

In this article we consider a more general situation, allowing degeneracy of critical orbits. To investigate the global bifurcation problem in this situation we apply an equivariant generalisation of the Rabinowitz alternative (see Theorem \ref{thm:RabAlt}) obtained by an application of the equivariant degree.
We also apply Dancer's result proposed in \cite{Dancer1984} and its generalisation from \cite{FRR} called the splitting lemma (see Lemma \ref{thm:splitting}). This lemma allows to separate the non-degenerate and degenerate parts of the problems and, under some standard assumptions, it allows to reduce comparing the degrees to comparing its non-degenerate parts.
In that way we obtain global bifurcations of solutions of the system on a sphere (Theorem \ref{thm:sphere}) and on a ball (Theorem \ref{thm:ball}). These theorems are the main results of our paper.

These two results are of different types. This is a consequence of the differences between the spectral behaviours of the Laplacians on their domains. In the case of the system on a sphere the only radial solutions are the constant ones. From this we can conclude a necessary condition for a bifurcation, see Lemma \ref{lemma:warunekkonieczny}. This allows to indicate the exact levels of global bifurcations of solutions of the system \eqref{eq:sphere11} in Theorem \ref{thm:sphere}. In the case of a ball there are nonconstant radial solutions and therefore the reasoning from the proof of Lemma \ref{lemma:warunekkonieczny} cannot be applied to obtain its counterpart for the system \eqref{eq:ball11}. Hence the bifurcation result given in Theorem \ref{thm:ball} does not provide such a precise indication and it is formulated in a different way, namely, there is given only an approximate location of the bifurcation.

We emphasise that, as far as we know, the bifurcation problem in the degenerate situation in elliptic systems has not yet been investigated, even in the case of critical sets consisting of isolated points.

Our main tool to obtain the bifurcation results is the equivariant degree.
We define an index of an isolated critical orbit to be the degree on some neighbourhood of this orbit.
To obtain the results for the differential systems we have developed the methods of computing this index, generalising the methods from \cite{GolSte}.
In Theorem \ref{thm:degrees} we have obtained a relation of the index of an orbit with the index of an isolated point from the space normal to this orbit.
Using this result, in the case of admissible pairs of groups, we have simplified the comparison of the indices of orbits by reducing it to the comparison of the indices of the  critical points, see Corollary \ref{cor:differentfinitedegrees}.

The abstract results described above are proved in Section \ref{sec:degree}, while Section \ref{sec:elliptic} is devoted to study elliptic systems \eqref{eq:sphere11} and \eqref{eq:ball11}. At the end of the paper, for the convenience of the reader, we recall some relevant material in the appendix.


\section{Degree of a critical orbit}\label{sec:degree}

In this section we introduce the topological method that we use to prove the bifurcation, i.e. a method of computing the index of an isolated critical orbit. We allow the case when such an orbit is degenerate. The index is defined in terms of  a degree for equivariant gradient maps, computed in some neighbourhood of the orbit. Note that we recall the notion of the degree in the appendix. Since the bifurcation phenomenon occurs when the index changes, our aim is to formulate results which allow to compare the degrees.

Let $G$ be a compact Lie group, $\bV$ - a finite dimensional orthogonal $G$-representation and $\phi \in C^2(\bV,\bR)$ - a $G$-invariant function. Fix $v_0 \in (\nabla \phi)^{-1}(0)$ and consider the orbit $G(v_0)$ of $v_0$.  It is known, that $G$-invariance of $\phi$ implies $G$-equivariance of $\nabla \phi$, hence $G(v_0) \subset (\nabla \phi)^{-1}(0)$. Assume that this orbit is isolated in such a set. Therefore, we can choose a $G$-invariant open set $\Omega \subset \bV$ such that $(\nabla \phi)^{-1}(0)\cap cl(\Omega) = G(v_0)$, which implies that $\phi$ is $\Omega$-admissible. Moreover, without loss of generality, we can assume that $\Omega = G \cdot B_{\epsilon}(v_0,\bW)$, where  $\bW=T_{v_0}^{\perp}G(v_0)$ is the space normal to the orbit $G(v_0)$ at $v_0$, $B_{\epsilon}(v_0,\bW)$ is the open ball in $\bW$ of radius $\epsilon$, centred at $v_0$, and $\epsilon$ is given by the slice theorem (see Theorem \ref{thm:slice}).

Under the above assumptions, the degree $\degg(\nabla \phi, \Omega)$ is well-defined and our aim is to obtain conditions simplifying the comparison of the degrees of such form. The result of this kind has been given in \cite{GolSte}. However, it can be applied only in the case when critical orbits of $\phi$ are non-degenerate. In this section we allow the situation when $G(v_0)$ is degenerate.

In this case to compute the degree we use its definition, in particular the fact that it can be given as the degree of an associated $\Omega$-Morse function (see the Appendix for the definition). More precisely, Lemma \ref{lem:istnienie-Morsa} implies the existence of a $G$-invariant open set $\Omega_0 \subset \Omega$ and a $G$-invariant $\Omega$-Morse function $\widehat{\phi}$ which coincide with $\phi$ outside the set $\Omega_0$. Therefore $\widehat{\phi}$ and $\phi$ are $\Omega$-homotopic, so the homotopy invariance of the degree implies that
\begin{equation}\label{eq:deg1}
\degg(\nabla \phi, \Omega)=\degg(\nabla \widehat{\phi}, \Omega).
\end{equation}

Since $\widehat{\phi}$ is an $\Omega$-Morse function, it has a finite number of critical orbits in $\Omega$. In the next lemma we show that the isotropy groups of elements of critical orbits of $\widehat{\phi}$ are related to the isotropy group of $v_0$. Recall that by $G_{v_0}$ we denote the isotropy group of $v_0$ and put $H=G_{v_0}$. It is known that $\bW$ is an orthogonal representation of $H$. With the above notation there holds the following result:

\begin{Lemma}\label{lem:orbittypes}
If $v\in(\nabla \widehat{\phi})^{-1}(0) \cap \Omega$, then $G_v$ is conjugate  to a subgroup of $H$.
\end{Lemma}
\begin{proof}

Fix $v\in(\nabla \widehat{\phi})^{-1}(0) \cap \Omega$. Then, since $\Omega = G \cdot B_{\epsilon}(v_0,\bW)$,
there exist $g\in G$ and $w\in B_{\epsilon}(v_0,\bW)$ such that $v=gw$. From Theorem \ref{thm:slice} it follows that $gw=\theta([g,w])$, where $\theta$ is a $G$-diffeomorphism.  Hence $G_{gw}=G_{[g,w]}$. To finish the proof we use the fact that $G_{[g,w]}=gH_wg^{-1}$ (see Lemma 4.16 of \cite{Kawakubo}). Obviously, $H_w \subset H$.
\end{proof}

With functions $\phi$ and $\widehat{\phi}$ we consider their restrictions to the space $\bW$. Consider an $H$-invariant function $\psi \colon \bW \to \bR$ given by $\psi=\phi|_{\bW}$. From the definitions of $\psi$ and $\Omega$ it follows that $(\nabla \psi)^{-1}(0)\cap cl(B_{\epsilon}(v_0,\bW)) = H(v_0) =\{v_0\}$ and therefore $\psi$ is a $B_{\epsilon}(v_0,\bW)$-admissible function.

In a similar way consider an $H$-invariant function $\widehat{\psi}\colon \bW \to \bR$ given by $\widehat{\psi}=\widehat{\phi}|_{\bW}$. In the following we will show that $\widehat{\psi}$ is an $H$-invariant $B_{\epsilon}(v_0,\bW)$-Morse function associated with $\psi$.
First of all note that from the definitions of $\widehat{\psi}$ and $\psi$ it follows that $\widehat{\psi}(w)=\psi(w)$ for every $w \notin \Omega_0 \cap \bW$.

Denote by $G(v_1), \ldots, G(v_q)$ the set of all the critical orbits of $\widehat{\phi}$ in $\Omega$.  In the next lemma we show that all critical orbits of $\widehat{\psi}$ in $B_{\epsilon}(v_0,\bW)$ can be obtained from this set.

\begin{Lemma}
There exist  $w_1, \ldots, w_q \in B_{\epsilon}(v_0,\bW)$ such that $G(v_i)\cap \bW = H(w_i)$, for $i=1,\ldots, q$, and $(\nabla \widehat{\psi})^{-1}(0)\cap B_{\epsilon}(v_0,\bW)=H(w_1) \cup \ldots \cup H(w_q)$.
\end{Lemma}

\begin{proof}

Fix $i \in \{1, \ldots, q\}$ and observe that from the definition of $\Omega$ there exist $g \in G, w_i \in B_{\epsilon}(v_0,\bW)$ such that $v_i=gw_i$.
 Therefore $w_i \in G(v_i) \cap \bW$ and hence $H(w_i) \subset G(v_i) \cap \bW$.
 To prove the opposite inclusion assume that $\tilde gv_i\in \bW$ for some $\tilde g \in G$. This implies that $\tilde gv_i\in \Omega \cap \bW=B_{\epsilon}(v_0,\bW).$ On the other hand, taking again $g \in G$ and $w_i \in \bW$ such that $v_i=gw_i$ we obtain that $\tilde{g}gw_i\in B_{\epsilon}(v_0,\bW).$ By Theorem \ref{thm:slice2}, this implies that $\tilde{g} g\in H$. Hence $G(v_i) \cap \bW= H(w_i)$.

To prove that $H(w_1) \cup \ldots \cup H(w_q)\subset (\nabla \widehat{\psi})^{-1}(0)\cap B_{\epsilon}(v_0,\bW)$ note that, by the orthogonality of $\bW$ as an $H$-representation, if $w_i\in B_{\epsilon}(v_0,\bW)$, then $H(w_i)\subset B_{\epsilon}(v_0,\bW)$ and, since $hw_i\in G(v_i)$ for every $h\in H$,
$\nabla \widehat{\psi}(hw_i)=\nabla \widehat{\phi}(hw_i)=0.$
Therefore $H(w_i) \subset (\nabla \widehat{\psi})^{-1}(0)\cap B_{\epsilon}(v_0,\bW)$ for every $i=1,\ldots, q$.

To finish the proof note that if $w \in (\nabla\widehat{\psi})^{-1}(0) \cap B_{\epsilon}(v_0,\bW)$, then obviously $w \in (\nabla\widehat{\phi})^{-1}(0)$. In particular, for some $i=1,\ldots, q$, $w\in G(w_i)\cap B_{\epsilon}(v_0,\bW)$, which is equal to  $H(w_i)$. Therefore $(\nabla \widehat{\psi})^{-1}(0)\cap B_{\epsilon}(v_0,\bW) \subset H(w_1) \cup \ldots \cup H(w_q)$.
\end{proof}

 In the following lemma we prove that all the critical orbits of $\widehat{\psi}$ in $B_{\epsilon}(v_0,\bW)$ are non-degenerate.

\begin{Lemma}
Let $H(w_1), \ldots, H(w_q)$ be the orbits obtained in the previous lemma. Then for every $i=1,\ldots,q$, the orbit $H(w_i)$ is a non-degenerate critical orbit of $\widehat{\psi}$.
\end{Lemma}

\begin{proof}
Suppose that the orbit $H(w_i)$ is a degenerate critical orbit of $\widehat{\psi}$, i.e. $\ker \nabla^2 \widehat{\psi}(w_i)=T_{w_i} H(w_i)\oplus \bU_i$, where $\bU_i$ is a linear space such that $\dim \bU_i>0$. Then, since $T_{v_0}G(v_0)\oplus T_{w_i}H(w_i)=T_{w_i}G(w_i)$, we obtain $$\ker \nabla^2 \widehat{\phi}(w_i)=T_{v_0} G(v_0)\oplus \ker \nabla^2 \widehat{\psi}(w_i)= T_{v_0} G(v_0) \oplus T_{w_i} H(w_i)\oplus \bU_i =T_{w_i}G(w_i)\oplus \bU_i,$$
which contradicts non-degeneracy of $G(w_i)$ as a critical orbit of $\widehat{\phi}$.
\end{proof}

Therefore, we have proved that $\widehat{\psi}$ is an $H$-invariant $B_{\epsilon}(v_0,\bW)$-Morse function associated with $\psi.$ Hence,
\begin{equation}\label{eq:deg2}
\degh(\nabla \psi, B_{\epsilon}(v_0,\bW))=\degh(\nabla \widehat{\psi}, B_{\epsilon}(v_0,\bW)).
\end{equation}

From equations \eqref{eq:deg1} and \eqref{eq:deg2} we see that in order to obtain a relation of equivariant degrees of $\nabla \phi$ and $\nabla \psi$ we need to establish such a relation for $\nabla \widehat{\phi}$ and $\nabla \widehat{\psi}$.

Such a result has been given in \cite{GolSte} (as a part of the proof of Theorem 2.6). Note that we can apply this result, since $\widehat{\phi}$ is the $G$-extension of $\widehat{\psi}$, given by $\widehat{\phi}(gw)=\widehat{\psi}(w)$ for $w \in \bW, g \in G$.  From the slice theorem it follows that the $G$-extension to the space $G \cdot B_{\epsilon}(v_0,\bW)$ is unique, see Theorem \ref{thm:uniq}.

\begin{Theorem}\label{thm:degrees}
Let $\widehat{\phi},\widehat{\psi}$ and $\Omega$ be defined as above. Denote by  $(H_{i_0})_H, (H_{i_1})_H, \ldots, (H_{i_s})_H$ all possible conjugacy classes of the groups $H_{w_1}, \ldots, H_{w_q}$ in $H$ and let $m_1,\ldots, m_s$ be such that $$\degh(\nabla \widehat{\psi}, B_{\epsilon}(v_0,\bW))=\sum_{j=1}^s m_j \cdot\chi_H(H/H_{i_j}^+) \in U(H).$$
Then $$\degg(\nabla \widehat{\phi}, \Omega)=\sum_{(K)_G \in \overline{sub}[G]} n_{(K)} \cdot \chi_G(G/K^+) \in U(G),$$
where $$n_{(K)}= \sum_{(H_{w_i})_G=(K)_G} m_j$$
and $\overline{sub}[G]$ is the set of conjugacy classes of closed subgroups of $G$.
\end{Theorem}

Now we turn our attention to the problem of comparing degrees. Consider two $G$-invariant  functions $\phi_1, \phi_2 \in C^2(\bV, \bR)$ with critical orbits $G(v_i)$ isolated in $(\nabla \phi_i)^{-1}(0)$ for $i=1,2$. Suppose that $G_{v_1}=G_{v_2}= H$. In the corollary below we study the special case when $(G,H)$ is an admissible pair. The concept of an admissible pair has been introduced in \cite{PRS}. The pair $(G,H)$, where $H$ is a subgroup of $G$, is called admissible, if the following condition is satisfied: for any subgroups $H_1,H_2$ of $H$ if $H_1$ and $H_2$ are not conjugate in $H$, then they are not conjugate in $G$.
It is worth to point out, that in the considered applications to elliptic systems the admissibility assumption is satisfied in a natural way.

\begin{Corollary}\label{cor:differentfinitedegrees}
Consider $G$-invariant functions $\phi_i \in C^2(\bV, \bR)$ with critical orbits $G(v_i)$ isolated in $(\nabla \phi_i)^{-1}(0)$ for $i=1,2$. Suppose that $G_{v_1}=G_{v_2}= H$ and
put $\bW_i=T_{v_i}^{\bot} G(v_i)$. Fix $\Omega_i=G\cdot B_{\epsilon}(v_i,\bW_i)$ such that $(\nabla\phi_i)^{-1}(0)\cap \Omega_i= G(v_i)$ and $\epsilon$ satisfies Theorem \ref{thm:slice}. Define $\psi_i\colon B_{\epsilon}(v_i,\bW_i)\to\bR$ by $\psi_i=\phi_{i|B_{\epsilon}(v_i,\bW_i)}$.
If $(G,H)$ is an admissible pair and
\begin{equation}\label{eq:finite1}
\nabla_H\text{-}\deg(\nabla\psi_1,B_{\epsilon}(v_1,\bW_1))\neq \nabla_H\text{-}\deg(\nabla\psi_2,B_{\epsilon}(v_2,\bW_2)),
\end{equation}
then
\begin{equation}\label{eq:finite2}
\nabla_G\text{-}\deg(\nabla\phi_1, \Omega_1)\neq \nabla_G\text{-}\deg(\nabla\phi_2, \Omega_2).
\end{equation}
\end{Corollary}

\begin{Remark}
With similar assumptions, a counterpart of the above result can be also formulated in the infinite dimensional case. Then $\phi_i, \psi_i$ instead of being finite dimensional maps can be completely continuous perturbations of the identity or, in a more general setting, completely continuous perturbations of some Fredholm operators. In the former case one can apply the degree for equivariant gradient maps in infinite dimensional spaces (see Appendix), whereas in the latter case the degree for invariant strongly indefinite functionals (see \cite{GolRyb2011}) can be used.
However, since in the proofs of our results we restrict our attention to the finite dimensional maps, we do not give the precise formulation of the infinite dimensional result.
Such a relation in the non-degenerate infinite dimensional case has been formulated in Theorem 2.8 of \cite{GolSte}.
\end{Remark}


\section{Elliptic systems}\label{sec:elliptic}
\subsection{Formulation of the problem}
Our aim in this section is to study bifurcations of weak solutions of some nonlinear problems, parameterised by $\lambda \in \bR$. We consider two different types of problems, namely the system defined on the $(N-1)$-dimensional unit sphere $S^{N-1}$:
\begin{equation}\label{eq:sphere}
- \Delta u =  \nabla_u F(u, \lambda) \quad \text{ on } \quad S^{N-1}
\end{equation}
and the Neumann problem on the $N$-dimensional open unit ball $B^{N}$:
\begin{equation}\label{eq:ball}
\left\{
 \begin{array}{rclcl}   -\triangle u & =& \nabla_u F(u,\lambda )   & \text{ in   } & B^N \\
                   \frac{\partial u}{\partial \nu} & =  & 0 & \text{ on    } & S^{N-1}.
\end{array}\right.
\end{equation}
We consider these systems with analogous assumptions. Moreover, the general setting in both problems is similar. Therefore, in the rest of this subsection, we describe both cases simultaneously.

We assume that the considered systems satisfy the following conditions:
\begin{enumerate}
\item[(B1)] $\bR^p$ is an orthogonal representation of a compact Lie group $\Gamma.$
\item[(B2)] $F\in C^2(\bR^p\times\bR,\bR)$ is $\Gamma$-invariant with respect to the first variable, i.e.  $F(\gamma u,\lambda) =F(u,\lambda)$ for every $\gamma \in \Gamma$, $u\in\bR^p$, $\lambda\in\bR$. Moreover,
there exist $C>0$ and $s\in [1,(N+2)(N-2)^{-1})$ such that $|\nabla^2_u F(u,\lambda)|\leq C(1+|u|^{s-1})$ (if $N=2$, we assume that $s\in[1,+\infty))$.
\item[(B3)] $u_0\in\bR^p$ is a critical point of $F(\cdot,\lambda)$ for all $\lambda \in \bR$ and there exists a symmetric matrix $A$ such that $\nabla_u^2 F(u_0,\lambda)=\lambda A$,
\item[(B4)] $\Gamma_{u_0}=\{e\}$,
\end{enumerate}
From (B2)-(B3) it follows that $\Gamma(u_0)  \subset (\nabla_u F(\cdot,\lambda))^{-1}(0)$ for every $\lambda \in \bR$.
We assume:
\begin{enumerate}
\item[(B5)] $\Gamma(u_0)\times \bR$ is isolated in the set of all critical points of $F$.
\item[(B6)] $\deg_B(\nabla_u F|_{T_{u_0}^{\perp}\Gamma(u_0)}(\cdot,\lambda), B_{\epsilon}(T_{u_0}^{\perp}\Gamma(u_0)),0)\neq 0$ for every $\lambda\in\bR\setminus\{0\}$ and sufficiently small $\epsilon$, where $\deg_B$ denotes the Brouwer degree.
\end{enumerate}

\begin{Remark}
The assumption (B6) is given in terms of the Brouwer degree of some map on a space normal to the orbit. It is worth to point out that in some cases this assumption is easy to verify. For example, it is satisfied in the case of a non-degenerate critical orbit or if the orbit consists of minima of the potential, see for example the proofs of Theorems 5.6-5.7 of \cite{Strzelecki}. Additionally, it is easy to see that in the former case the assumption (B5) is also satisfied.
\end{Remark}

We are going to study bifurcations of non-constant solutions of these systems, by applying the results from the previous section. To this end we will formulate the problem in a variational setting, i.e. we will consider weak solutions as critical points of some associated functionals. We define such functionals on  appropriate Hilbert spaces. More precisely, consider $\cM\in\{B^N,S^{N-1}\}$, denote by $H^1(\cM)$ the Sobolev space on $\cM$ with the inner product
\begin{equation*} \langle v, w \rangle_{H^1(\cM)} =   \int\limits_{\cM} (\nabla v(x), \nabla w(x)) + v(x)\cdot w(x) dx
\end{equation*}
and consider a separable Hilbert space $\bH=
\bigoplus_{i=1}^{p} H^1(\cM)$ with the scalar product
\begin{equation}\label{iloczyn} \langle v, w \rangle_{\bH} = \ds \sum_{i=1}^p \langle v_i, w_i\rangle_{H^1(\cM)}.
\end{equation}

Weak solutions of the problems \eqref{eq:sphere} and \eqref{eq:ball} are in one-to-one correspondence with critical points of the functional $\Phi \colon \bH \times \bR \to \bR$
defined by
\begin{equation}\label{eq:Phi}
\Phi (u,\lambda) = \frac{1}{2} \int\limits_{\cM} |\nabla u(x)|^2 dx- \int\limits_{\cM} F(u(x),\lambda)dx.
\end{equation}

Denote by $\tilde{u}_0\in\bH$ the constant function $\tilde{u}_0(x)\equiv u_0$.
A standard computation (see for example \cite{GolKlu} for a similar reasoning) shows that the gradient of the functional $\Phi$ has the following form
$$\nabla_u \Phi (u, \lambda)= u - \tilde{u}_0 - L_{\lambda A}(u-\tilde{u}_0)+ \nabla \eta(u-\tilde{u}_0,\lambda),$$
where $L_{\lambda A}\colon\bH \to \bH$ is given by
\begin{equation*}
\langle L_{\lambda A} u, v \rangle_{\h} = \int\limits_{\cM} (u(x), v(x)) + (\lambda A u(x), v(x)) dx \text{ for all } v\in\bH. \end{equation*}
 Moreover $L_{\lambda A}$ is a self-adjoint, bounded and completely continuous operator
and $\nabla_u\eta\colon\bH\times\bR\to\bH$ is a completely continuous operator such that $\nabla_u\eta(0,\lambda)=0$, $\nabla^2_u\eta(0,\lambda)=0$ for every $\lambda\in\bR$.

The systems \eqref{eq:sphere} and \eqref{eq:ball} have two kinds of symmetries - the $\Gamma$-invariance of the potential $F$ and the $SO(N)$-invariance of the domain $\cM$.
These symmetries are inherited by the associated functional. More precisely, it is easy to check that the space $\bH$, with the product given by \eqref{iloczyn}, is an orthogonal representation of $G = \Gamma \times SO(N)$, where the action is  given by
\begin{equation}\label{eq:action}
(\gamma, \alpha) (u)(x)= \gamma u({\alpha}^{-1}x)\  \text{ for }\  (\gamma, \alpha) \in G, u \in \h, x\in \cM.
\end{equation}
This definition and the assumption (B2) imply that the functional $\Phi$ is $G$-invariant.

Since $u_0 \in (\nabla_u F(\cdot,\lambda))^{-1}(0)$, the function $\tilde{u}_0$ is a solution of the problem \eqref{eq:sphere} (respectively the problem \eqref{eq:ball}) for all $\lambda \in \bR$.
By the definition of the action of the group $G$ on $\bH$ given by \eqref{eq:action} we obtain that $G(\tilde{u}_0)=\Gamma(\tilde{u}_0)$ is an orbit of weak solutions for all $\lambda \in \mathbb{R}$. We are going to study bifurcations from this family, i.e. from $G(\tilde{u}_0) \times \bR$.

\begin{Definition}\label{def:globbif}
We say that a global bifurcation of solutions of $\nabla_u \Phi(u,\lambda)=0$ occurs from the orbit $G(\tilde{u}_0)\times \{\lambda_0\}$ if there is a connected component $\cC(\tilde{u}_0,\lambda_0)$ of $$cl \{(u, \lambda) \in (\bH \times \bR) \setminus (G(\tilde{u}_0) \times \bR)\colon \nabla_u\Phi(u, \lambda)=0\}$$ containing $(\tilde{u}_0,\lambda_0)$ and such that either $\cC(\tilde{u}_0,\lambda_0)\cap (G(\tilde{u}_0) \times (\bR\setminus \{\lambda_0\})) \neq \emptyset$ or $\cC(\tilde{u}_0,\lambda_0)$ is unbounded.
\end{Definition}

Note that if a global bifurcation occurs, we obtain a connected component of the set of nontrivial solutions for every $v\in G(\tilde{u}_0)$ (by nontrivial solutions we understand solutions different than those from the family $G(\tilde{u}_0) \times \bR$).
In particular, if the group $G$ is connected, we obtain a connected set of nontrivial solutions bifurcating from $G(\tilde{u}_0)\times \{\lambda_0\}$.

To investigate the global bifurcation problem we apply the degree for equivariant gradient maps, see the appendix. From the above reasoning it follows that $\nabla_u \Phi (\cdot, \lambda)$ is $G$-equivariant and of the form of a completely continuous perturbation of the identity for every $\lambda\in\bR$. Therefore to be able to apply the degree it only remains to define a $G$-equivariant approximation scheme on the space $\bH$. This scheme is defined with the use of the eigenspaces of the Laplacian, therefore  we first recall some basic spectral properties of this operator.

Denote by $\sigma(-\Delta;\cM) = \{ 0=\beta_1 < \beta_2 < \ldots < \beta_k < \ldots\}$ the set of all distinct eigenvalues of the Laplacian on $\cM$. Let $\bV_{-\Delta}(\beta_k)$ be the eigenspace of $-\Delta$ corresponding to an eigenvalue $\beta_k$.
From the spectral theorem it follows that $H^1(\cM) = cl(\bigoplus_{k=1}^{\infty} \bV_{-\Delta} (\beta_k)).$
Therefore it is natural to define the approximation scheme in the following way: consider $\bH^n=\bigoplus_{k=1}^n \bH_k$, where $\bH_k=\bigoplus_{i=1}^p \bV_{-\Delta}(\beta_k)$ and a natural $G$-equivariant projection $\pi_n\colon\bH \to \bH$ such that $\pi_n(\bH) = \bH^n$ for $n\in \bN$.
Then $\{\pi_n\colon \bH \to \bH\colon n \in \bN\}$ defines a $G$-equivariant approximation scheme on $\bH.$

\begin{Remark}\label{rem:linearequation}
A standard computation shows that
\[
\sigma(Id-L_{\lambda A})=\left\{
 \frac{\beta_k-\lambda\alpha_j}{1+\beta_k}\colon \alpha_j \in\sigma(A), \beta_k\in \sigma(-\Delta; \cM)\right\},
\]
where $\sigma(\cdot)$ denotes the spectrum of a linear operator. For the details we refer the reader for example to the proof of Lemma 3.2 in \cite{GolKlu}.
\end{Remark}

In the proofs of the main theorems of this section we consider the eigenspaces of the Laplacian as representations of the group $SO(N)$.
We recall the characterisation of their nontriviality in the following lemma.

\begin{Lemma}\label{lem:irred}
Fix $\beta\in\sigma(-\Delta;\cM)$.
\begin{enumerate}
\item If $\cM=S^{N-1}$, then for all $\beta\neq 0$ the eigenspaces $\bV_{-\Delta}(\beta)$ are irreducible nontrivial $SO(N)$-representations and $\bV_{-\Delta}(0)$ is a trivial one.
\item If $\cM=B^N$ and $\dim\bV_{-\Delta}(\beta)>1$, then the eigenspace $\bV_{-\Delta}(\beta)$ is a nontrivial $SO(N)$-representation. If  $\dim\bV_{-\Delta}(\beta)=1$, then the corresponding eigenspace is trivial. \end{enumerate}
\end{Lemma}
For the proof in the case of a sphere we refer to \cite{Gurarie} (Theorem 5.1), in the second one to Remark 5.11 of \cite{GolKluSte}.

\subsection{Global bifurcations of solutions of the system on a sphere}\label{subsec:sphere}

From now on we consider the systems separately, starting with the one on a sphere:
\begin{equation} \label{eq:sphere1}
   - \Delta u =  \nabla_u F(u, \lambda) \quad \text{ on } \quad S^{N-1}
\end{equation}
under the assumptions (B1)-(B6).
For such a system we can describe exactly all the levels where a bifurcation can occur.
Such a description is given in the lemma below. Its proof is similar in spirit to the one of Theorem 3.2.1 of \cite{PRS2} and it is based on the fact that the only radial eigenfunctions of the Laplacian on a sphere are constant functions.

Put
\[
\Lambda = \bigcup_{\alpha_j \in \sigma(A) \setminus\{0\}} \bigcup_{\beta_k \in \sigma(-\Delta;S^{N-1})\setminus \{0\}} \left\{\frac{\beta_k}{\alpha_j}\right\}.
\]

\begin{Lemma}\label{lemma:warunekkonieczny}
 If $(\tilde{u}_0, \lambda_0)$ is an accumulation point of nontrivial solutions of the system \eqref{eq:sphere1}, then $\lambda_0 \in \Lambda$.
\end{Lemma}

\begin{proof}
Fix $\lambda_0\in\bR$ and suppose that $(\tilde{u}_0, \lambda_0)$ is an accumulation point of nontrivial weak solutions of the system \eqref{eq:sphere1}, i.e. there exists a sequence of nontrivial solutions $(v_n, \lambda_n)$ converging to $(\tilde{u}_0, \lambda_0)$.

By the definition of the functional associated to the system, weak solutions of \eqref{eq:sphere1} are solutions of
\begin{equation}\label{eq:cale}
\nabla_u\Phi(u, \lambda)=0.
\end{equation}
This equation is equivalent to the system
\begin{equation}\label{eq:ker}
\nabla_{u_1}\Phi(u_1,u_2, \lambda)=0,
\end{equation}
\begin{equation}\label{eq:im}
\nabla_{u_2}\Phi(u_1,u_2, \lambda)=0,
\end{equation}
where $u=(u_1,u_2)\in \ker\nabla_u^2\Phi(\tilde{u}_0,\lambda_0) \oplus \im\nabla_u^2\Phi(\tilde{u}_0,\lambda_0)=\bH$.

Considering the latter equation, due to the fact that $\nabla_{u_2}^2\Phi(\tilde{u}_0,\lambda_0)$ is an isomorphism, we can use the equivariant implicit function theorem (see \cite{FRR}). This theorem implies the existence of an $SO(N)$-equivariant map
$$\omega\colon D_\varepsilon(\ker\nabla_u^2\Phi(\tilde{u}_0,\lambda_0))\times(\lambda_0-\varepsilon,\lambda_0+\varepsilon)\to \im\nabla_u^2\Phi(\tilde{u}_0,\lambda_0)$$ such that $(u_1,u_2,\lambda)=(u_1,\omega(u_1,\lambda),\lambda)$ is the only solution of \eqref{eq:im} at a neighbourhood  of $(\tilde{u}_0,\lambda_0)$.
Therefore the solutions of \eqref{eq:cale} in such a neighbourhood must be of the form $(u_1,\omega(u_1,\lambda),\lambda)$.

Consider the isotropy group of such a solution. Using the $G$-equivariance of $\omega$, we obtain $G_{(u_1, w(u_1, \lambda),\lambda)}=G_{u_1}\cap G_{w(u_1, \lambda)}\cap G_{\lambda}=G_{u_1}$. Therefore the existence of a sequence $(v_n, \lambda_n)$ converging to $(\tilde{u}_0, \lambda_0)$ implies that \begin{equation}\label{eq:nonempty}
\ker\nabla_u^2\Phi(\tilde{u}_0,\lambda_0)\cap\bigoplus_{k=2}^{\infty}\bH_k \neq \emptyset.
\end{equation}
Indeed, assume, by contrary, that $\ker\nabla_u^2\Phi(\tilde{u}_0,\lambda_0)\subset\bH_1$ and consider the group action \eqref{eq:action}, restricted to the subgroup $\{e\} \times SO(N)$, on the spaces $\bH_k$. Identifying $\{e\} \times SO(N)$ with $SO(N)$, from Lemma \ref{lem:irred} we have $\bH_1=\bH^{SO(N)}$. Therefore the above computations, with the assumption $\ker\nabla_u^2\Phi(\tilde{u}_0,\lambda_0)\subset\bH_1$, imply that  $SO(N)_{(v_n,\lambda_n)}=SO(N)$ for every $n$. As a consequence, $v_n\in \bH_1$, i.e. $(v_n,\lambda_n)$ are constant solutions corresponding to critical points of $F$. This is a contradiction with (B5), i.e. the fact that $\Gamma(u_0)\times \bR$ is isolated in the set of critical points of $F$.

To finish the proof, notice that the inequality \eqref{eq:nonempty} and Remark \ref{rem:linearequation} imply that there are $\beta_k\in \sigma(-\Delta; S^{N-1})\setminus\{0\}$ and $\alpha_j \in \sigma(A)$ such that $\frac{\beta_k-\lambda_0\alpha_j}{1+\beta_k}=0$, i.e. $\beta_k=\lambda_0\alpha_j$. This in particular proves that $\alpha_j \neq 0$ and finally $\lambda_0 = \frac{\beta_k}{\alpha_j}\in\Lambda$.
\end{proof}

\begin{Remark}\label{rem:loc}
Lemma \ref{lemma:warunekkonieczny} gives us a necessary condition for the so called local bifurcation. This bifurcation does not have to be a global one, i.e.  the bifurcating set does not have to be connected. \end{Remark}

Now we are in a position to prove the sufficient condition for the global bifurcation phenomenon of nontrivial solutions of the system \eqref{eq:sphere1}. Namely, we will show that for all $\lambda\in\Lambda$ the phenomenon occurs from the orbit $G(\tilde{u}_0) \times \{\lambda\}$. In other words, we will prove that for the system on a sphere, the necessary condition is also a sufficient one.

\begin{Theorem}\label{thm:sphere}
Consider the system \eqref{eq:sphere1} with the potential $F$ and $u_0$ satisfying the assumptions (B1)-(B6) and fix $\lambda_0 \in \Lambda$. Then
a global bifurcation of solutions of \eqref{eq:sphere1} occurs from the orbit $G(\tilde{u}_0) \times \{\lambda_0\}$.
\end{Theorem}

\begin{proof}
Fix $\lambda_0 \in \Lambda$ and choose $\varepsilon>0$  such that $\Lambda\cap[\lambda_0-\varepsilon,\lambda_0+\varepsilon]=\{\lambda_0\}$.
From the definition of $\Lambda$ such a choice is always possible.
Since $\lambda_0 \pm \varepsilon \notin \Lambda,$ Lemma \ref{lemma:warunekkonieczny} implies that $G(\tilde{u}_0) \subset \bH$ is an isolated critical orbit of the $G$-invariant functionals $\Phi(\cdot, \lambda_0 \pm \varepsilon) \colon \bH \to \bR.$ Therefore there exists an open, bounded and $G$-invariant subset $\Omega\subset \bH$  such that $(\nabla_u \Phi (\cdot, \lambda_0\pm\varepsilon))^{-1}(0)\cap cl(\Omega)=G(\tilde{u}_0).$

Due to the global bifurcation theorem (see Theorem \ref{thm:RabAlt}), and taking into consideration the necessary condition of bifurcation (Lemma \ref{lemma:warunekkonieczny}), to prove the assertion it is enough to show that
\begin{equation}\label{eq:DegreeS}
\nabla_{G}\textrm{-}\mathrm{deg}(\nabla_u\Phi(\cdot,\lambda_0-\varepsilon), \Omega) \neq
\nabla_{G}\textrm{-}\mathrm{deg}(\nabla_u\Phi(\cdot,\lambda_0+\varepsilon), \Omega).
\end{equation}
From the definition of the degree this inequality is equivalent with
\begin{equation}\label{eq:DegreeOrtS}
\nabla_{G}\textrm{-}\mathrm{deg}(\nabla_u\Phi|_{\bH^n}(\cdot,\lambda_0-\varepsilon),  \Omega\cap\bH^n) \neq
\nabla_{G}\textrm{-}\mathrm{deg}(\nabla_u\Phi|_{\bH^n}(\cdot,\lambda_0+\varepsilon), \Omega\cap\bH^n),
\end{equation}
where $n$ is sufficiently large. In the following we assume that this $n$ is fixed.

Denote by $H$ the isotropy group of $\tilde{u}_0$ and note that by the assumption (B4) we have $H=\{e\}\times SO(N)$.
Consider an $H$-representation $\bW=T_{\tilde{u}_0}^{\perp}G(\tilde{u}_0)$ and an $H$-invariant functional $\Psi=\Phi|_{\bW}$.
By Corollary \ref{cor:differentfinitedegrees} (note that the pair $(G,H)=(\Gamma \times SO(N),\{e\}\times SO(N))$ is admissible, see Lemma 2.8 of \cite{GolKluSte}) and the excision property of the degree, to prove \eqref{eq:DegreeOrtS} it is enough to show that
$$\nabla_{H}\textrm{-}\mathrm{deg}(\nabla_u\Psi|_{\bW^n}(\cdot,\lambda_0-\varepsilon),  B_{\delta}(\tilde{u}_0,\bW^n)) \neq
\nabla_{H}\textrm{-}\mathrm{deg}(\nabla_u\Psi|_{\bW^n}(\cdot,\lambda_0+\varepsilon), B_{\delta}(\tilde{u}_0,\bW^n)),$$
where $\bW^n = \bW \cap \bH^n$ and $B_{\delta}(\tilde{u}_0,\bW^n)\subset \Omega\cap\bW^n$ is such that $\delta>0$ is taken from the slice theorem (see Theorem \ref{thm:slice}).

Instead of the functional $\Psi|_{\bW^n}$ it is more convenient to consider the shifted functional $\Pi^n(u, \lambda)=\Psi|_{\bW^n}(u+\tilde{u}_0, \lambda)$. Denote by $\Pi_{\pm}(u)=\Pi^n(u, \lambda_0 \pm \varepsilon)$ and note that (by (B5)) $0$~is an isolated critical point of $\Pi_{\pm}$. Moreover
$$\nabla_{H}\textrm{-}\mathrm{deg}(\nabla_u\Psi|_{\bW^n}(\cdot,\lambda_0\pm\varepsilon),  B_{\delta}(\tilde{u}_0,\bW^n))=
\nabla_{H}\textrm{-}\mathrm{deg}(\nabla\Pi_{\pm},  B_{\delta}(\bW^n)).$$

Put $\cL_{\pm}=\nabla^2\Pi_{\pm}(0)\colon \bW^n \to \bW^n$. Since $\lambda_0 \pm \varepsilon \notin \Lambda$, we obtain $\ker(\cL_{-})=\ker(\cL_{+})$ and, by the self-adjointness, $\im(\cL_{-})=\im(\cL_{+})$. We denote the kernel and image of $\cL_{\pm}$ by  $\cN$ and $\cR$, respectively.

The computation of the degree of $\nabla \Pi_{\pm}$ can be reduced to the computation of the degrees of some maps defined on $B_{\delta}(\cN)$ and $B_{\delta}(\cR)$. More precisely, using the splitting lemma (see Theorem \ref{thm:splitting}) we obtain $\epsilon>0$ and $H$-equivariant homotopies $\nabla_u \cH_{\pm} \colon  (B_\epsilon(\cN)\times B_\epsilon(\cR))\times[0,1]\to\bW^n$ connecting $\nabla \Pi_{\pm}$ with a product mapping $(\nabla\varphi_{\pm},((\cL_{\pm})|_{\cR}) ),$ where  $\nabla\varphi_{\pm} \colon B_{\epsilon}(\cN) \to \cN$ are some $H$-equivariant maps. Without loss of generality we can assume that $\epsilon=\delta$.

Therefore by the homotopy invariance property of the degree we obtain
$$\nabla_{H}\textrm{-}\mathrm{deg}(\nabla\Pi_{\pm}, B_{\delta}(\bW^n)) =\nabla_{H}\textrm{-}\mathrm{deg}((\nabla\varphi_{\pm},(\cL_{\pm})|_{\cR} ), B_{\delta}(\cN) \times B_{\delta}(\cR)).$$

Note that since $\cL_{\pm}$ are isomorphisms, the $(B_{\delta}(\cN) \times B_{\delta}(\cR))$-admissibility of the product maps described above implies the $B_{\delta}(\cN)$-admissibility of $\nabla \varphi_{\pm}$. Therefore we can apply the product formula (see Lemma \ref{lem:prod}) of the degree and obtain
\begin{equation*}
\begin{split}&\nabla_{H}\textrm{-}\mathrm{deg}((\nabla\varphi_{\pm},(\cL_{\pm})|_{\cR}), B_{\delta}(\cN) \times B_{\delta}(\cR))=\\&=\nabla_{H}\textrm{-}\mathrm{deg}(\nabla\varphi_{\pm},B_\delta(\cN)) \star \nabla_{H}\textrm{-}\mathrm{deg}((\cL_{\pm})|_{\cR},B_\delta(\cR)).
\end{split}
\end{equation*}

Denote by $\cT$ the space normal to the orbit $\Gamma(\tilde{u}_0)$ in $\bH_1$ (i.e. $\cT=T_{\tilde{u}_0}^{\perp} \Gamma(\tilde{u}_0)$ where the complement is taken in $\bH_1$).
Note that, since $\lambda_0\pm\varepsilon \notin \Lambda$, we have $\cN\subset \bH_1$ and $\cR = (\cT \ominus \cN) \oplus \bigoplus_{k=2}^n \bH_k$.
Put $\widetilde{\cR}=\cT\ominus \cN$  and $\cR_n=\bigoplus_{k=2}^n \bH_k.$

Using again the fact that $\cL_{\pm}$ are isomorphisms we can apply once more the product formula, obtaining
 $$\ds \nabla_{H}\textrm{-}\mathrm{deg}((\cL_{\pm})|_{\cR},B_\delta(\cR))=
\nabla_{H}\textrm{-}\mathrm{deg}((\cL_{\pm})|_{\widetilde\cR},B_\delta(\widetilde\cR)) \star \nabla_{H}\textrm{-}\mathrm{deg}((\cL_{\pm})|_{\cR_n},B_\delta(\cR_n))$$
and consequently
\begin{equation}\label{eq:Pi}
\begin{split} &\nabla_{H}\textrm{-}\mathrm{deg}(\nabla\Pi_{\pm}, B_{\delta}(\bW^n))=\\
&=\nabla_{H}\textrm{-}\mathrm{deg}(\nabla\varphi_{\pm},B_\delta(\cN)) \star
\nabla_{H}\textrm{-}\mathrm{deg}((\cL_{\pm})|_{\widetilde\cR},B_\delta(\widetilde\cR)) \star \nabla_{H}\textrm{-}\mathrm{deg}((\cL_{\pm})|_{\cR_n},B_\delta(\cR_n))\\
&=\nabla_{H}\textrm{-}\mathrm{deg}((\nabla\varphi_{\pm},(\cL_{\pm})|_{\widetilde\cR}) ,B_\delta(\cN) \times B_\delta(\widetilde\cR) ) \star \nabla_{H}\textrm{-}\mathrm{deg}((\cL_{\pm})|_{\cR_n},B_\delta(\cR_n)).
\end{split}
\end{equation}

Let us consider the first factor in the above product. It is easy to check that the restrictions of the homotopies $\cH_{\pm}$ (given by the splitting lemma, see also Remark \ref{rem:homotopia}) to $\cT=\cN \oplus \widetilde{\cR}$ connect $(\nabla \varphi_{\pm}, (\cL_{\pm})|_{\widetilde\cR})$ with $\nabla \Psi|_{\cN \oplus\widetilde{\cR}}$.
Moreover, from the definitions of $\Psi$ and $\cT$ we have $\nabla \Psi|_{\cT}=\nabla F|_{\cT}$.
From the above, since $\cN \oplus \widetilde{\cR} \subset \bW^H$, using equation \eqref{eq:Brouwer} we obtain
$$\nabla_{H}\textrm{-}\mathrm{deg}((\nabla\varphi_{\pm},(\cL_{\pm})|_{\widetilde\cR}) ,B_\delta(\cN) \times B_\delta(\widetilde\cR) )= \deg_B(\nabla F|_{\cT}( \cdot, \lambda_0 \pm \varepsilon), B_{\delta}(\cT),0)\cdot \bI,$$ where $\bI$ is the identity in $U(H)$.

Consider now the latter factor in the last product in \eqref{eq:Pi}. Denote by $\cW(\lambda), \cV(\lambda)$ the negative and zero eigenspaces of $(Id-L_{\lambda A})|_{\cR_n}$ (i.e. the direct sum of the eigenspaces of $(Id-L_{\lambda A})|_{\cR_n}$, corresponding to negative and zero eigenvalues, respectively). The description of these spaces can be obtained via Remark \ref{rem:linearequation}. Moreover, from this description it follows that for $n$ sufficiently large $\cW(\lambda)$ and $\cV(\lambda)$ are the negative and zero eigenspaces of $(Id-L_{\lambda A})|_{\bH_1^{\perp}}$. Without loss of generality we can assume that $n$ satisfies this condition. Therefore, from Remark \ref{rem:linearequation}, we have
$$\cW(\lambda)=
\bigoplus_{\alpha_j\in\sigma( A)} \ \bigoplus_{\substack{\beta_k \in \sigma(-\Delta;S^{N-1})\setminus\{0\}\\
\beta_k<\lambda  \alpha_j}} \bV_{-\Delta}(\beta_k)^{\mu_{ A}(\alpha_{j})}$$
and
\begin{equation*}
\cV(\lambda)=\bigoplus_{\alpha_j \in \sigma(A)} \bigoplus_{\substack{\beta_k \in \sigma(-\Delta;S^{N-1})\setminus\{0\}\\ \beta_k =\lambda \alpha_j}} \bV_{-\Delta}(\beta_k)^{\mu_A(\alpha_j)}.
\end{equation*}
Here $\mu_A(\alpha_j)$ denotes the multiplicity of $\alpha_j$ as an eigenvalue of $A$. Moreover $\bV_{-\Delta}(\beta_k)^{\mu_A(\alpha_j)}$ is formally understood as $\mathrm{span}\{h \cdot f\colon h \in \bV_{-\Delta}(\beta_k), f \in \bV_A(\alpha_j)\}$, see also \cite{GolKluSte}. However, since in our computations only the dimensions of these spaces are important, we use the fact that this space is isomorphic to the direct sum of $\mu_A(\alpha_j)$ copies of $\bV_{-\Delta}(\beta_k).$

Consider $\lambda_0 > 0 $ and assume that $\varepsilon$ is such that $\lambda_0 -\varepsilon >0.$

A standard computation, see for example \cite{GolKlu}, shows that
$$\nabla_{H}\textrm{-}\mathrm{deg}((\cL_{+})|_{\cR_n},B_\delta(\cR_n))
=\nabla_{H}\textrm{-}\mathrm{deg}(-Id,B_\delta(\cW(\lambda_0 - \varepsilon))) \star
\nabla_{H}\textrm{-}\mathrm{deg}(-Id,B_\delta(\cV(\lambda_0)))
$$
and
$$\nabla_{H}\textrm{-}\mathrm{deg}((\cL_{-})|_{\cR_n},B_\delta(\cR_n))
= \nabla_{H}\textrm{-}\mathrm{deg}(-Id,B_\delta(\cW(\lambda_0 - \varepsilon))).$$
Moreover, it is known that the above degrees of $-Id$ are invertible (see Theorem 2.1  of \cite{GolRyb2011}).

Summing up, using equation \eqref{eq:Pi} and the above computations, we obtain that to prove \eqref{eq:DegreeB} we have to show
$$
\deg_B(\nabla F|_{\cT}( \cdot, \lambda_0 + \varepsilon), B_{\delta}(\cT),0)\cdot \nabla_{H}\textrm{-}\mathrm{deg}(-Id,B_\delta(\cV(\lambda_0))) \neq\deg_B(\nabla F|_{\cT}( \cdot, \lambda_0 -\varepsilon), B_{\delta}(\cT),0)\cdot \bI.
$$

Since $\cV(\lambda_0)$ is a nontrivial $SO(N)$-representation (see Lemma \ref{lem:irred}) and therefore also a nontrivial $H$-representation, by Remark \ref{rem:nontriv},
$\nabla_{H}\textrm{-}\mathrm{deg}(-Id,B_\epsilon(\cV(\lambda_0)))\neq a \cdot \bI$
for any $a\in\bZ$. Using the assumption (B6) we obtain the assertion in this case.

Similarly, for $\lambda_0 <0$, assuming $\lambda_0 + \varepsilon <0$  we obtain
$$\nabla_{H}\textrm{-}\mathrm{deg}((\cL_{+})|_{\cR_n},B_\delta(\cR_n))
=\nabla_{H}\textrm{-}\mathrm{deg}(-Id,B_\delta(\cW(\lambda_0 + \varepsilon)))
$$
and
$$
\nabla_{H}\textrm{-}\mathrm{deg}((\cL_{-})|_{\cR_n},B_\delta(\cR_n))
= \nabla_{H}\textrm{-}\mathrm{deg}(-Id,B_\delta(\cW(\lambda_0 + \varepsilon)))
\star
\nabla_{H}\textrm{-}\mathrm{deg}(-Id,B_\delta(\cV(\lambda_0))).
$$
The rest of the proof is analogous to the previous case.
\end{proof}

\subsection{Global bifurcations of solutions of the system on a ball}\label{ssec:ball}

Now we turn our attention to the system defined on a ball:
\begin{equation}\label{eq:ball1}
\left\{
 \begin{array}{rclcl}   -\triangle u & =& \nabla_u F(u,\lambda )   & \text{ in   } & B^N \\
                   \frac{\partial u}{\partial \nu} & =  & 0 & \text{ on    } & S^{N-1}
\end{array}\right.
\end{equation}
under the assumptions (B1)-(B6). For the Laplacian considered on a ball there exist radial eigenfunctions different than the constant ones. Consequently, the method used to prove the necessary condition for a bifurcation in the case of the system considered in Subsection \ref{subsec:sphere} cannot be applied to find a counterpart here. Therefore we will study bifurcations of solutions of \eqref{eq:ball1} in a different way.
In particular for some $\lambda\in \bR$ we are able to prove only a local (not global) bifurcation (we say that a local bifurcation occurs from the orbit $G(\tilde{u}_0) \times \{\lambda_0\}$ if $(\tilde{u}_0,\lambda_0)$ is an accumulation point of nontrivial solutions of \eqref{eq:ball1}).

\begin{Theorem}\label{thm:ball}
Consider the system \eqref{eq:ball1} with the potential $F$ and $u_0$ satisfying the assumptions (B1)-(B6). Let $\alpha_0 \in \sigma(A) \setminus\{0\}$ and $\beta_0 \in \sigma(-\Delta;B^N)\setminus \{0\}$ be such that $\dim \bV_{-\Delta}(\beta_0)>1$. Then for every $\varepsilon>0$, such that $0\notin(\frac{\beta_0}{\alpha_0}-\varepsilon,\frac{\beta_0}{\alpha_0}+\varepsilon)$, at least one of the following statements holds:
\begin{enumerate}[(i)]
\item  a local bifurcation of solutions of \eqref{eq:ball1} occurs from the orbit $G(\tilde{u}_0) \times \{\lambda\}$ for every $\lambda\in(\frac{\beta_0}{\alpha_0}-\varepsilon,\frac{\beta_0}{\alpha_0})$ or for every $\lambda\in(\frac{\beta_0}{\alpha_0},\frac{\beta_0}{\alpha_0}+\varepsilon)$,
\item a global bifurcation of solutions of \eqref{eq:ball1} occurs from the orbit $G(\tilde{u}_0) \times \{\widehat\lambda\}$ for some $\widehat\lambda \in (\frac{\beta_0}{\alpha_0}-\varepsilon,\frac{\beta_0}{\alpha_0}+\varepsilon)$.
\end{enumerate}
\end{Theorem}

\begin{proof}
Fix $\varepsilon>0$ such that $0\notin(\frac{\beta_0}{\alpha_0}-\varepsilon,\frac{\beta_0}{\alpha_0}+\varepsilon)$ and assume that the first statement does not hold. Then there exist $\lambda_{\pm}$ such that  $$\frac{\beta_0}{\alpha_0}-\varepsilon<\lambda_-<\frac{\beta_0}{\alpha_0}<\lambda_+< \frac{\beta_0}{\alpha_0}+\varepsilon$$
and $G(\tilde{u}_0)$ is an isolated critical orbit of $\Phi(\cdot, \lambda_{\pm}) \colon \bH \to \bR$.
Therefore there exists $\Omega\subset \bH$ being an open, bounded and $G$-invariant subset such that $(\nabla_u \Phi (\cdot, \lambda_{\pm}))^{-1}(0)\cap cl(\Omega)=G(\tilde{u}_0)$. Hence the degrees $\nabla_{G}\textrm{-}\mathrm{deg}(\nabla_u\Phi(\cdot,\lambda_{\pm}), \Omega)$ are well-defined and to prove the assertion we will apply the Rabinowitz alternative given in Theorem \ref{thm:RabAlt}. To this end we will show that
\begin{equation}\label{eq:DegreeB}
\nabla_{G}\textrm{-}\mathrm{deg}(\nabla_u\Phi(\cdot,\lambda_-), \Omega) \neq
\nabla_{G}\textrm{-}\mathrm{deg}(\nabla_u\Phi(\cdot,\lambda_+), \Omega).
\end{equation}

To prove \eqref{eq:DegreeB} consider an orthogonal $H$-representation $\bW=T_{\tilde{u}_0}^{\perp}G(\tilde{u}_0)$ and an $H$-invariant functional $\Psi=\Phi|_{\bW\times \bR}$, where $H=\{e\}\times SO(N)$ is the isotropy group of $\tilde{u}_0$.
Then, as in the proof of Theorem \ref{thm:sphere}, we obtain that \eqref{eq:DegreeB} is equivalent with
\begin{equation}\label{eq:DegreeBrest}
\nabla_{H}\textrm{-}\mathrm{deg}(\nabla_u\Psi|_{\bW^n}(\cdot,\lambda_-),B_{\delta}(\tilde{u}_0,\bW^n))
\neq
\nabla_{H}\textrm{-}\mathrm{deg}(\nabla_u\Psi|_{\bW^n}(\cdot,\lambda_+), B_{\delta}(\tilde{u}_0,\bW^n)),
\end{equation}
where $\bW^n = \bW \cap \bH^n$. Considering the shifting of the critical point to the origin we get
$$\nabla_{H}\textrm{-}\mathrm{deg}(\nabla_u\Psi|_{\bW^n}(\cdot,\lambda_{\pm}),B_{\delta}(\tilde{u}_0,\bW^n))=
\nabla_{H}\textrm{-}\mathrm{deg}(\nabla\Pi_{\pm}, B_{\delta}(\bW^n)),$$
where $\Pi_{\pm}(u)=\Psi|_{\bW^n}(u+\tilde{u}_0,\lambda_{\pm})$.

It is easy to observe that from the assumptions it follows that $\lambda_+$ and $\lambda_-$ are of the same sign. Suppose that $\lambda_+ >\lambda_->0$.
Reasoning as in the proof of Theorem \ref{thm:sphere}, we can obtain formulae for $\nabla_{H}\textrm{-}\mathrm{deg}(\nabla\Pi_{\pm}, B_{\delta}(\bW^n))$. More precisely, for $\lambda\in\bR$ let us denote
$$\cW(\lambda)=
\bigoplus_{\alpha_j\in\sigma( A)} \ \bigoplus_{\substack{\beta_k \in \sigma(-\Delta;B^{N})\setminus\{0\}\\\beta_k<\lambda  \alpha_j}} \bV_{-\Delta}(\beta_k)^{\mu_{ A}(\alpha_{j})}.
$$
Moreover, put
\[
\Lambda = \bigcup_{\alpha_j \in \sigma(A) \setminus\{0\}} \bigcup_{\beta_k \in \sigma(-\Delta;B^N)\setminus \{0\}} \left\{\frac{\beta_k}{\alpha_j}\right\}
\]
and $$
\cV_{(\lambda_-,\lambda_+)}=\bigoplus_{\lambda\in\Lambda \cap(\lambda_-,\lambda_+)} \bigoplus_{\alpha_j \in \sigma(A)} \bigoplus_{\substack{\beta_k \in \sigma(-\Delta;B^{N})\setminus\{0\}\\ \beta_k =\lambda \alpha_j}} \bV_{-\Delta}(\beta_k)^{\mu_A(\alpha_j)}.
$$
Then, applying the splitting lemma with the homotopy invariance property and the product formula, we obtain
\begin{multline*}
\nabla_{H}\textrm{-}\mathrm{deg}(\nabla\Pi_-,B_{\delta}(\bW^n))
=\deg_B(\nabla F|_{T_{u_0}^{\perp}\Gamma(u_0)}(\cdot,\lambda_-), B_{\epsilon}(T_{u_0}^{\perp}\Gamma(u_0)),0)\cdot\\
\cdot \nabla_{H}\textrm{-}\mathrm{deg}(-Id,B_\epsilon(\cW(\lambda_-)))
\end{multline*}
and
\begin{multline*}
\nabla_{H}\textrm{-}\mathrm{deg}(\nabla\Pi_+, B_{\delta}(\bW^n))
= \deg_B(\nabla F|_{T_{u_0}^{\perp}\Gamma(u_0)}(\cdot,\lambda_+), B_{\epsilon}(T_{u_0}^{\perp}\Gamma(u_0)),0)\cdot \\
\cdot \nabla_{H}\textrm{-}\mathrm{deg}(-Id,B_\epsilon(\cW(\lambda_-)))
 \star \nabla_{H}\textrm{-}\mathrm{deg}(-Id,B_\epsilon(\cV_{(\lambda_-,\lambda_+)})).
\end{multline*}

Hence, by the invertibility of $\nabla_{H}\textrm{-}\mathrm{deg}(-Id,B_\epsilon(\cW(\lambda_-)))$, see Theorem 2.1 of \cite{GolRyb2011}, we can reduce \eqref{eq:DegreeBrest} to
\begin{multline*}
\deg_B(\nabla F|_{T_{u_0}^{\perp}\Gamma(u_0)}(\cdot,\lambda_-), B_{\epsilon}(T_{u_0}^{\perp}\Gamma(u_0)),0) \cdot \bI \neq
\\ \neq
\deg_B(\nabla F|_{T_{u_0}^{\perp}\Gamma(u_0)}(\cdot,\lambda_+), B_{\epsilon}(T_{u_0}^{\perp}\Gamma(u_0)),0) \cdot
 \nabla_{H}\textrm{-}\mathrm{deg}(-Id,B_\epsilon(\cV_{(\lambda_-,\lambda_+)})).
\end{multline*}

To finish the proof note that $\bV_{-\Delta}(\beta_0)$ is a nontrivial $SO(N)$-representation (by Lemma \ref{lem:irred}). Therefore, since  $\bV_{-\Delta}(\beta_0)\subset \cV_{(\lambda_-,\lambda_+)}$, the space $\cV_{(\lambda_-,\lambda_+)}$ is also a nontrivial $SO(N)$-representation.
Hence, by Remark \ref{rem:nontriv},
$\nabla_{H}\textrm{-}\mathrm{deg}(-Id,B_\epsilon(\cV_{(\lambda_-\lambda_+)}))\neq a \cdot \bI$
for any $a\in\bZ$.
Using (B6) we therefore obtain \eqref{eq:DegreeBrest}, from this and Corollary \ref{cor:differentfinitedegrees} we get \eqref{eq:DegreeB}.
Applying Theorem \ref{thm:RabAlt} we complete the proof in the case of positive $\lambda_{\pm}$.

Reasoning in a similar way in the case $\lambda_{-}<\lambda_+<0$ we finish the proof.
\end{proof}

\begin{Remark}
In \cite{GolKluSte} we have investigated a problem similar to \eqref{eq:ball1}. Namely, we have considered the system \eqref{eq:ball1} with the potential of the form $F(u,\lambda)=\lambda f(u)$ with $f$ having a non-degenerate critical orbit. In particular we have proved that in this non-degenerate case, under the assumptions of Theorem \ref{thm:ball}, there occurs a global bifurcation exactly from the orbit $G(\tilde{u}_0) \times \{\frac{\beta_0}{\alpha_0}\}$. Theorem \ref{thm:ball} therefore generalises this result to the case of a degenerate critical orbit and a more general potential $F$.

The precise indication of the bifurcation level is a consequence of the necessary condition given in Lemma 3.1 of \cite{GolKluSte}. Since we do not have such a condition in the degenerate case, we have obtained only an approximate location of the bifurcation, i.e. we have proved that there is a global bifurcation at any arbitrarily small neighbourhood of $\frac{\beta_0}{\alpha_0}$ or there occurs a local bifurcation at every level from this neighbourhood.
\end{Remark}

\section{Appendix}\label{sec:appendix}

\subsection{Equivariant topology}

Let $\bV$ be a finite dimensional, orthogonal representation of a compact Lie group $G$. Fix $v_0\in \bV$, put $H=G_{v_0}$ and consider an $H$-representation $\bW=T_{v_0}^{\bot} G(v_0)$, i.e $\bW$ is the space normal to the orbit $G(v_0)$ at $v_0$.
Denote by $B_{\epsilon}(v_0,\bW)$ the open ball of radius $\epsilon$ centred at $v_0$ and by $G\times_H X$ the twisted product of an $H$-space $X$ over $H$, see \cite{Kawakubo}, \cite{TomDieck}.
In the theorems below we collect some properties of a neighbourhood of an orbit, for their proofs see for example \cite{Bredon}, \cite{DuisKolk}, \cite{Field}, \cite{Mayer}.

\begin{Theorem}[Slice theorem]\label{thm:slice}
There exists $\epsilon>0$ such that the mapping $G\times_H \bW \to \bV$ defined by $[g,w]\mapsto gw$ induces a $G$-equivariant diffeomorphism $\theta$ from $G\times_H B_{\epsilon}(v_0,\bW)$ to an open $G$-invariant neighbourhood  $G\cdot B_{\epsilon}(v_0,\bW)=\{g w\colon g\in G, w\in B_{\epsilon}(v_0,\bW) \}$ of the orbit $G(v_0)$.
\end{Theorem}

\begin{Theorem}\label{thm:slice2}
If $\epsilon$ is given by Theorem \ref{thm:slice} and $g\cdot B_{\epsilon}(v_0,\bW)\cap B_{\epsilon}(v_0,\bW) \neq \emptyset$, then $g\in H$.
\end{Theorem}

\begin{Theorem}\label{thm:uniq}
Fix $G$-invariant functions $\phi_1,\phi_2\colon G\cdot B_{\epsilon}(v_0,\bW)\to\bR$, where $\epsilon$ is given by Theorem \ref{thm:slice} and assume that $\phi_1(w)=\phi_2(w)$ for every $w\in B_{\epsilon}(v_0,\bW)$. Then $\phi_1=\phi_2$.
\end{Theorem}

In our paper we use a topological invariant (namely the degree for equivariant gradient maps) which is an element of the Euler ring $(U(G), +, \star)$, see \cite{TomDieck1}, \cite{TomDieck} for the definition of this ring. When applying the degree we use the fact that $U(G)$ can be identified with the $\bZ$-module $\bigoplus_{(H)_G \in \overline{sub}[G]} \bZ$ (see Corollary IV.1.9 of \cite{TomDieck}), where $\overline{sub}[G]$ is the set of conjugacy classes of closed subgroups of $G$. Moreover, we use the representation of elements of $U(G)$ as finite sums of the form $\sum_{(H)_G \in \overline{sub}[G]} n_H \cdot \chi_G(G/H^+)$, where $n_H \in \bZ$, $\chi_G(G/H^+)$ is a $G$-equivariant Euler characteristic of a pointed $G$-CW-complex $G/H^+$ (see \cite{TomDieck}). The unit in $U(G)$ is $\bI=\chi_G(G/G^+)$.

\subsection{Equivariant degree}
Let $\bV$ be a finite dimensional, orthogonal representation of a compact Lie group $G$ and let $\varphi \in C^1(\bV, \bR)$ be a $G$-invariant function. Moreover, let $\Omega \subset \bV$ be an open, bounded $G$-invariant set such that $\partial \Omega \cap (\nabla \varphi)^{-1}(0) = \emptyset.$
In such a case we say that $\varphi$ is an $\Omega$-admissible function.
For such $\bV, \varphi$ and $\Omega$ G\c{e}ba has defined in \cite{Geba} the degree $\nabla_G\textrm{-}\mathrm{deg}(\nabla \varphi, \Omega)$, being an element of the Euler ring $U(G)$.

The definition given by Gęba uses the fact that any $G$-invariant $\Omega$-admissible function can be approximated by a $G$-invariant $\Omega$-Morse function and in the next step by a so called special $G$-invariant $\Omega$-Morse function. However, in our applications it is enough to use the relation of $G$-invariant $\Omega$-admissible functions with $G$-invariant $\Omega$-Morse functions. Recall that $\varphi$ is called an $\Omega$-Morse function if for every $v\in(\nabla \varphi)^{-1}(0) \cap \Omega$ the orbit $G(v)$ is non-degenerate, i.e. $\dim \ker \nabla^2 \varphi (v)= \dim G(v)$.

The following lemma is a consequence of Theorem 2.2 of \cite{Fang}.

\begin{Lemma} \label{lem:istnienie-Morsa}
If $\varphi \in C^2(\bV, \bR)$ is an $\Omega$-admissible $G$-invariant function, then there exists a $G$-invariant set $\Omega_0$ and a $G$-invariant function  $\widehat{\varphi} \in C^2(\bV, \bR)$ such that
\begin{enumerate}
\item $(\nabla \varphi)^{-1}(0) \cap \Omega\subset \Omega_0\subset cl(\Omega_0)\subset \Omega,$
\item $\widehat{\varphi}(v)=\varphi(v)$ for every  $v \in \bV \setminus \Omega_0,$
\item $\widehat{\varphi}$ is a $G$-invariant $\Omega$-Morse function.
\end{enumerate}
\end{Lemma}

We say that a $G$-invariant $\Omega$-Morse function $\widehat{\varphi}$ satisfying the assertion of the above lemma is associated with $\varphi$. Note that from the lemma it follows that $\widehat{\varphi}$ and $\varphi$ are $\Omega$-homotopic, i.e. there exists a $G$-invariant $C^2$-function $ h\colon \bV \times [0,1] \to \bR$ such that  $(\nabla_v h)^{-1} (0) \cap (\partial \Omega \times [0,1])=\emptyset$ and $\nabla_v h(v, 0) = \nabla \varphi(v), \nabla_v h(v,1) = \nabla \widehat{\varphi}(v)$.

The degree $\nabla_G\textrm{-}\mathrm{deg}(\nabla \varphi, \Omega)$ has properties analogous to these of the Brouwer degree such as additivity, excision, linearisation, homotopy invariance, see \cite{Geba}, \cite{Ryb2005milano}.
Moreover, there holds the product formula for this degree, see \cite{GolRyb2013}. For the convenience of the reader we recall it below.

\begin{Lemma}\label{lem:prod}
Let $\Omega_i\subset\bV_i$ be open, bounded and $G$-invariant subsets of $G$-representations $\bV_i$ and let $\varphi_i\in C^1(\bV_i,\bR)$ be $G$-invariant and $\Omega_i$-admissible functions for $i=1,2$. Then $\varphi_1+\varphi_2\in C^1(\bV_1\oplus\bV_2,\bR)$ is $\Omega_1\times\Omega_2$-admissible and
 $$\nabla_G\text{-}\deg((\nabla\varphi_1,\nabla\varphi_2), \Omega\times\Omega_2)=\nabla_G\text{-}\deg(\nabla\varphi_1, \Omega_1)\star\nabla_G\text{-}\deg(\nabla\varphi_2, \Omega_2).$$
\end{Lemma}

\begin{Remark}\label{rem:nontriv}
It is known (see \cite{Geba}) that if $G$ acts trivially on $\Omega$, then
\begin{equation}\label{eq:Brouwer}
\nabla_G\textrm{-}\mathrm{deg}(\nabla \varphi, \Omega)=
                   \deg_B(\nabla \varphi, \Omega,0) \cdot \bI.
\end{equation}

On the other hand, applying the results from \cite{GarRyb} (in particular Lemma 3.4) it is easy to prove that for any nontrivial $G$-representation $\bV$ there holds
$\nabla_{G}\textrm{-}\mathrm{deg}(-Id,B(\bV))\neq a \cdot \bI$ for any $a\in\bZ$.
\end{Remark}

Now we are going to provide a sketch of the definition of the $G$-equivarint degree in the infinite dimensional case. Denote by $\bH$ an infinite dimensional, separable Hilbert space being an orthogonal representation of the group $G$.
Let $\{\pi_n\colon\bH \rightarrow  \bH\colon n \in \bN\}$ be a sequence of $G$-equivariant  orthogonal projections.
We say that this sequence is a $G$-equivariant approximation scheme on $\bH$ if the following conditions are fulfilled
\begin{enumerate}
\item $\bH^n = \im \pi_n$ is a finite dimensional, orthogonal $G$-representation for any $n\in \bN,$
\item $\bH^n \varsubsetneq \bH^{n+1}$ for any $n \in \bN,$
\item $\ds \lim_{n \to \infty} \pi_n u=u$ for any $u \in \bH$.
\end{enumerate}

Let us consider a $G$-equivariant gradient operator $\nabla\Phi \in C^1(\bH, \bH)$ of the form of a completely continuous perturbation of the identity.
Let $\Omega \subset \bH$ be an open, bounded $G$-invariant set such that $\partial \Omega \cap (\nabla \Phi)^{-1}(0)=\emptyset.$ In this situation the degree $\nabla_G\textrm{-}\mathrm{deg}(\nabla \Phi, \Omega) \in U(G)$ is defined by the following formula, see \cite{Ryb2005milano},
$$\nabla_G\textrm{-}\mathrm{deg}(\nabla \Phi, \Omega)=  \nabla_G\textrm{-}\mathrm{deg}(\nabla\Phi|_{\bH^n}, \Omega \cap \bH^n),$$
where $n$ is sufficiently large.

Our aim is to consider the global bifurcation phenomenon, see Definition \ref{def:globbif}. In the proofs of the bifurcation results in Section \ref{sec:elliptic} we apply an equivariant version of the Rabinowitz alternative, we recall it below.

Consider a family of $G$-invariant functionals $\Phi\in C^2(\bH\times\bR,\bR)$ of the form $\nabla_u \Phi(u,\lambda)=u-\nabla_u \zeta (u,\lambda),$ where
 $\nabla_u\zeta\colon\bH\times\bR\to\bH$ is a completely continuous, $G$-equivariant operator. Suppose that there is $u_0\in\bH$ such that  $G(u_0)\subset(\nabla_u \Phi(\cdot,\lambda))^{-1}(0)$ for every $\lambda\in\bR$. We call elements of $G(u_0)\times\bR$ the trivial solutions of $\nabla_u \Phi(u,\lambda)=0$.

\begin{Theorem}\label{thm:RabAlt}
Suppose that there are $\lambda_{\pm}$ and a $G$-invariant open bounded set $\Omega\subset\bH$ such that $(\nabla_u\Phi(\cdot,\lambda_{\pm}))^{-1}(0)\cap cl(\Omega)=G(u_0)$.
If
$$\nabla_G\text{-}\deg(\nabla_u\Phi(\cdot,\lambda_-),\Omega)\neq \nabla_G\text{-}\deg(\nabla_u\Phi(\cdot,\lambda_+),\Omega),$$
then a global bifurcation of solutions of $\nabla_u \Phi(u,\lambda)=0$ occurs from the orbit $G(u_0)\times \{\widehat\lambda\}$ for some $\widehat\lambda \in (\lambda_-,\lambda_+)$.
\end{Theorem}

The proof of this theorem is standard in the degree theory, see for instance \cite{Brown}, \cite{Dancer1973}, \cite{Rabinowitz}, \cite{Rabinowitz1}.

\subsection{Equivariant splitting lemma}
We end this section with recalling the so called splitting lemma that we use in computations of the degree in the degenerate case.

Consider a compact Lie group $H$, a finite dimensional orthogonal $H$-representation $\bW$ and an $H$-invariant function $\psi\in C^2(\bW,\bR)$. Suppose that 0 is its isolated critical point.
Assume additionally that $\nabla^2 \psi(0)$ is not an isomorphism.
Denote by $\cN$ and $\cR$ the kernel and the image of $\nabla^2 \psi(0)$, both being orthogonal $H$-representations.

\begin{Theorem}\label{thm:splitting}
There exist $\varepsilon>0$ and an $H$-equivariant homotopy $\nabla\cH\colon (B_\epsilon(\cN)\times B_\epsilon(\cR))\times[0,1]\to\bW$ satisfying
\begin{enumerate}
\item $(\nabla_u\cH)^{-1}(0)\cap ((B_\epsilon(\cN)\times B_\epsilon(\cR))\times[0,1])=\{0\}\times [0,1]$,  i.e. $0$ is an isolated critical point of $\cH(\cdot,t)$ for every $t\in[0,1]$,
\item $\nabla_u\cH((v,w),0)=\nabla\psi(v,w)$ for all $u =(v,w)\in B_\epsilon(\cN)\times B_\epsilon(\cR)$,
\item there exists an $H$-equivariant map $\nabla\varphi\colon B_\epsilon(\cN)\to \cN$ such that $\nabla_u\cH((v,w),1)=(\nabla\varphi(v),(\nabla^2\psi(0)|_{\cR}) w)$ for all $(v,w)\in B_\epsilon(\cN)\times B_\epsilon(\cR)$.
\end{enumerate}
\end{Theorem}

The proof of the above theorem can be found in \cite{FRR} (Lemma 3.2) and uses a homotopy proposed by Dancer in \cite{Dancer1984}.
Since in the proof of Theorem \ref{thm:sphere} the form of this homotopy is needed, we recall it below.

\begin{Remark}\label{rem:homotopia} The homotopy considered in Theorem \ref{thm:splitting} is of the form
\begin{multline*}\cH((v,w),t)=\frac{1}{2}\langle (\nabla^2\psi(0)|_{\cR}) w,w\rangle +\frac{1}{2} t (2-t)\langle (\nabla^2\psi(0)|_{\cR}) \widetilde{w}(v), \widetilde{w}(v)\rangle +\\+ t \eta(v,\widetilde{w}(v))+(1-t)\eta(v, w+t\widetilde{w}(v)\rangle,\end{multline*}
where $\psi(u) =\frac{1}{2} \langle \nabla^2\psi(0)(u),u\rangle + \eta(u)$ and $\widetilde{w}$ is a function obtained from the equivariant version of the implicit function theorem (see Theorem 3.1 of \cite{FRR}).
\end{Remark}


\end{document}